\documentclass[12pt]{amsart}
\usepackage{amsmath}
\usepackage{amssymb}
\usepackage{latexsym}
\usepackage{amscd}
\usepackage{citesort}
\usepackage{graphicx} 
\usepackage{amsthm}
\usepackage{mathrsfs}
\usepackage{xypic}
\usepackage{bm}

\newdimen\AAdi%
\newbox\AAbo%
\def\AAk#1#2{\s_etbox\AAbo=\hbox{#2}\AAdi=\wd\AAbo\kern#1\AAdi{}}%
\def\AAr#1#2#3{\s_etbox\AAbo=\hbox{#2}\AAdi=\ht\AAbo\raise#1\AAdi\hbox{#3}}%
\font\tenmsb=msbm10 at 12pt \font\sevenmsb=msbm7 at 8pt
\font\fivemsb=msbm5 at 6pt
\newfam\msbfam
\textfont\msbfam=\tenmsb \scriptfont\msbfam=\sevenmsb
\scriptscriptfont\msbfam=\fivemsb
\def\Bbb#1{{\tenmsb\fam\msbfam#1}}
\newtheorem{thm}{Theorem}[section]
\newtheorem{lem}{Lemma}[section]
\newtheorem{cor}{Corollary}[section]
\newtheorem{rem}{Remark}[section]
\newtheorem{pro}{Proposition}[section]

\newcommand{\ba}{\begin{array}}
\newcommand{\ea}{\end{array}}

\newcommand{\Section}[2]{\setcounter{equation}{0}
\allowdisplaybreaks
\section[#1]{#2}}

\def\n{\nabla}

\def\ir#1{\mathbb R^{#1}}

\def\f#1#2{\frac{#1}{#2}}

\def\grs#1#2{\bold G_{#1,#2}}

\def\dt#1{\frac {d\,#1}{d\,t}}

\def\td{\tilde}

\def\a{\alpha}
\def\be{\beta}

\def\p#1{\partial #1}

\def\de{\delta}
\def\De{\Delta}

\def\ep{\varepsilon}

\def\g{\gamma}

\def\la{\lambda}
\def\La{\Lambda}
\def\om{\omega}

\def\th{\theta}

\def\w{\wedge}

\def\D{\big(\f{d}{dt}-\De\big)}
\def\Hess{\mbox{Hess}}
\def\R{\Bbb{R}}
\def\tr{\mbox{tr}}
\def\ol{\overline}
\def\U{\Bbb{U}}
\def\V{\Bbb{V}}

\def\ra{\rightarrow}
\newcommand{\pB}[2]{\psi^{#1}|\n^{#2}B|^2}

\begin{document}
\title
[Mean curvature flow] {Mean curvature flow via convex functions
on Grassmannian manifolds}

\author
[Y.L. Xin and Ling Yang]{Y. L. Xin and Ling Yang}
\address
{Institute of Mathematics, Fudan University, Shanghai 200433,
China and Key Laboratory of Mathematics for Nonlinear Sciences
(Fudan University), Ministry of Education}
\email{ylxin@fudan.edu.cn}
\thanks{The research was partially supported by
NSFC (No. 10531090) and SFECC}
\begin{abstract}
Using the convex functions in Grassmannian manifolds we can carry
out interior estimates for mean curvature flow of higher
codimension. In this way some of the results in \cite{E-H2} can be
generalized to higher codimension
\end{abstract}

\renewcommand{\subjclassname}{%
  \textup{2000} Mathematics Subject Classification}
\date{}
\maketitle

\Section{Introduction}{Introduction}
\medskip

We consider the deformation of a complete submanifold in 
$\ir{m+n}$ under the mean curvature flow. For codimension one 
case there are many deep results given by Ecker-Huisken 
\cite{E-H1}\cite{E-H2}\cite{H1} and \cite{H2}.

In recent years some interesting work has been done for higher
codimensional mean curvature flow
\cite{C-L1}\cite{C-L2}\cite{C-T}\cite{S1}\cite{S2}\cite{S-W}\cite{W1} 
and  \cite{W2}. In a previous paper the first author studied mean 
curvature flow with convex Gauss image \cite{X3}. Some results in 
\cite{E-H1} has been generalized to higher codimensional 
situation. The present work would carry out interior estimates 
and generalize some results in \cite{E-H2} to higher codimension.

For a hypesurface there are  support functions which play an 
important role in gradient estimates for mean curvatute flow of 
codimension one. For general submanifolds we  can also define 
generalized support functions related to the generalized Gauss 
map whose image is the Grassmannian manifold. The Pl\"ucker 
imbedding of the Grassmannian manifold into Euclidean space gives 
us the "height functions" $w$ on the Grassmanian manifold. In the 
case of positive "height function" we can give lower bound of the 
Hessian of $\f{1}{w}$ in our previous paper \cite{X-Y}. Based on 
it we can define auxiliary functions which enable us to carry out 
gradient estimates for MCF in higher codimension from which we  
obtain confinable properties (Theorem 4.1) as well as curvature 
estimates (Theorem 5.1 and Theorem 5.2). In this way, we improve 
the previous results in \cite{X3}.

\Section{Convex functions on Grassmannian manifolds}{Convex
functions on Grassmannian manifolds}

Let $\R^{m+n}$ be an $(m+n)$-dimensional Euclidean space. All
oriented $n$-subspaces constitute the Grassmannian manifolds
$\grs{n}{m}.$

Fix $P_0\in \grs{n}{m}$ in the sequel, which is spanned by a unit
$n-$vector $\ep_1\w\cdots\w\ep_n$. For any $P\in\grs{n}{m}$,
spanned by an $n-$vector $e_1\w\cdots\w e_n$, we define an
important function on $\grs{n}{m}$,
$$w\mathop{=}\limits^{def.}\left<P, P_0\right>
=\left<e_1\wedge\cdots\wedge e_n, \ep_1\wedge\cdots\wedge
\ep_n\right>=\det W,$$ where $W=(\left<e_i, \ep_j\right>).$

Denote
\begin{equation*}
\U=\{P\in \grs{n}{m}:w(P)>0\}.
\end{equation*}
Let $\{\ep_{n+\a}\}$ be $m$ vectors such that
$\{\ep_i,\ep_{n+\a}\}$ form an orthornormal basis of $\ir{m+n}$.
Then we can span arbitrary $P\in \U$ by $n$ vectors $f_i$:
$$f_i=\ep_i+z_{i\a}\ep_{n+\a},$$
where $Z=(z_{i\a})$ are the local coordinates of $P$ in $\Bbb{U}$.
Here and in the sequel we use the summation convention and agree
the range of indices:
$$1\leq i,j\leq n;\qquad 1\leq \a,\be\leq m.$$
The Jordan angles between $P$ and $P_0$ are defined by
$$\th_\a=\arccos(\la_\a),$$
where $\la_\a\geq 0$ and $\la_\a^2$ are the eigenvalues of the
symmetric matrix $W^TW$ . On $\U$ we can define
$$v=w^{-1}.$$
Then it is easily seen that
\begin{equation*}\label{eq3}
v(P)=\big[\det(I_n+ZZ^T)\big]^{\f{1}{2}}=\prod_{\a=1}^m \sec\th_\a.
\end{equation*}

The canonical metric on $\grs{n}{m}$ in the local coordinates can 
be described as (see \cite{X1} Ch. VII)
\begin{equation}\label{metric}
g=\tr\big((I_n+ZZ^T)^{-1}dZ(I_m+Z^TZ)^{-1}dZ^T\big).
\end{equation}

Let $E_{i\a}$ be the matrix with 1 in the intersection of row $i$
and column $\a$ and 0 otherwise. Denote $g_{i\a,j\be}=\left<
E_{i\a},E_{j\be}\right>$ and let $\big(g^{i\a,j\be}\big)$ be the
inverse matrix of $\big(g_{i\a,j\be}\big)$. Then,
\begin{equation*}
(1+\la_i^2)^{\f{1}{2}}(1+\la_\a^2)^{\f{1}{2}}E_{i\a}
\end{equation*}
form an orthonormal basis of $T_P \grs{n}{m}$, where
$\la_\a=\tan\th_\a$. Denote its dual basis in $T_P^* \grs{n}{m}$
by $ \om_{i\a}.$

A lengthy computation yields \cite{X-Y}
\begin{eqnarray}\label{He2}\aligned
\Hess(v)_P&=\sum_{m+1\leq i\leq n,\a}v\ \om_{i\a}^2
+\sum_{\a}(1+\la_\a^2)v\ \om_{\a\a}^2 +v^{-1}\
dv\otimes dv\\
&\qquad\qquad+\sum_{\a<\be}\Big[(1+\la_\a\la_\be)v\Big(\f{\sqrt{2}}{2}(\om_{\a\be}+\om_{\be\a})\Big)^2\\
&\hskip1in+(1-\la_\a\la_\be)v\Big(\f{\sqrt{2}}{2}(\om_{\a\be}-\om_{\be\a})\Big)^2\Big].
\endaligned
\end{eqnarray}
Define
$$B_{JX}(P_0)=\big\{P\in \U:\mbox{ sum of any two Jordan angles}$$
$$\hskip2.4in\mbox{between }P\mbox{ and }P_0<\f{\pi}{2}\big\}.$$
This is a geodesic convex set, larger than the geodesic ball of
radius $\frac{\sqrt{2}}{4}\pi$ and centered at $P_0$. This  was
found in a  previous work of Jost-Xin \cite{J-X}. For any real
number $a$ let $\V_a=\{P\in\grs{n}{m},\quad v(P)<a\}.$ From
(\cite{J-X}, Theorem 3.2) we know that
$$\V_2\subset B_{JX}\qquad \text{and}\qquad
\ol{\V}_2\cap\ol{B}_{JX}\neq\emptyset$$

$\Hess(v)_P$ is positive definite if and only if
$\th_\a+\th_\be<\f{\pi}{2}$ for arbitrary $\a\neq \be$, i.e.,
$P\in B_{JX}(P_0)$.

From (\ref{He2}) it is easy to get an estimate
\begin{equation*}\label{es1}
\Hess(v)\geq v(2-v)g+v^{-1}dv\otimes dv\qquad \mbox{on }\ol{\V}_2.
\end{equation*}

For later applications the above estimate is not accurate enough.
Using the radial compensation technique the estimate could be
refined.

\begin{thm}\cite{X-Y}

$v$ is a convex function on $B_{JX}(P_0)\subset \U\subset
\grs{n}{m}$,  and
\begin{equation*}\label{es4}
\Hess(v)\geq
v(2-v)g+\Big(\f{v-1}{pv(v^{\f{2}{p}}-1)}+\f{p+1}{pv}\Big)dv\otimes
dv
\end{equation*}
on $\ol{\V}_2$, where $g$ is the metric tensor on $\grs{n}{m}$ and
$p=min(n,m)$.
\end{thm}

\begin{rem}
For any $a\le 2$, the sub-level set $\V_a$ is a convex set in
$\grs{n}{m}$.
\end{rem}

\begin{rem}
The sectional curvature varies in $[0,2]$ under the canonical
Riemannian metric. By the standard Hessian comparison theorem we
have
$$\text{Hess}(\rho)\ge \sqrt{2}\,\cot(\sqrt{2}\rho)(g-d \rho\otimes d
\rho),$$ where $\rho$ is the distance function from a fixed point
in $\grs{n}{m}$.
\end{rem}

\Section{Evolution equations}{Evolution equations}
\medskip

Let $M$ be a complete  $n-$submanifold in $\ir{m+n}.$ Consider the
deformation of $M$ under the mean curvature flow, i.e. $\exists$ a
one-parameter family $F_t=F(\cdot, t)$ of immersions
$F_t:M\to\ir{m+n}$ with corresponding images $M_t=F_t(M)$ such 
that
\begin{equation}\begin{aligned}
\dt{}F(x, t)&=H(x, t),\quad x\in M\\
F(x, 0)&=F(x),\end{aligned} \label{mcf}
\end{equation}
where $H(x, t)$ is the mean curvature vector of $M_t$ at  $F(x, t).$

From  equation (\ref{mcf}) it is easily known that
\begin{equation}
\left(\dt{}-\De\right)|F|^2=-2n.
\end{equation}

Let $B$ denote the second fundamental form of $M_t$ in $\ir{m+n}$.
It satisfies the evolution equation
\begin{lem}(Lemma 3.1 in \cite{X3})
\begin{equation}
\left(\dt{}-\De\right)|B|^2\le - \,2\,|\n|B||^2 + 3|B|^4\label{B}.
\end{equation}
\end{lem}

The  Gauss map $ \gamma : M \to \grs{n}{m} $ is defined by
$$
 \g (x) = T_x M \in \grs{n}{m}
$$
via the parallel translation in $ \ir{m+n} $ for $ \forall x \in
M $. The Gauss maps under the MCF satisfies the following
relation.
\begin{pro}\cite{W2}\label{wang2}
\begin{equation}
\dt{\g}=\tau(\g(t)),\label{prv}
\end{equation}
where $\tau(\g(t))$ is the tension fields of the Gauss map from
$M_t$.
\end{pro}

Let $h:\V\to\ir{}$ be a smooth function defined on an open subset
$\V\subset G_{n,m}$ and denote $\td{h}=h\circ \g$, then
\begin{equation*}
\f{d\td{h}}{dt}=\f{d(h\circ\g)}{dt}=dh\big(\tau(\g)\big).
\end{equation*}
On the other hand, by the composition formula
\begin{equation*}
\De \td{h}=\De(h\circ \g)=\Hess(h)(\g_* e_i,\g_* e_i)\circ
\g+dh\big(\tau(\g)\big),
\end{equation*}
where $\{e_i\}$ is a local orthonormal frame field on $M_t$; and
then we derive
\begin{equation}\label{j}
\big(\f{d}{dt}-\De\big)\td{h}=-\Hess(h)(\g_* e_i,\g_* e_i)\circ
\g.
\end{equation}

\Section{Confinable properties}{Confinable properties}
\medskip

Now, we consider the convex  Gauss image situation which is
preserved under the flow, so called confinable property.

Let $r:\R^{n+m}\times \R\ra \R$ be a smooth, nonnegative function,
such that for any $R>0$,
\begin{equation*}
\ol{M}_{t,R}=\big\{x\in M_t: r(x,t)\leq R^2\big\}
\end{equation*}
is compact.

\begin{lem}\label{l1}
Assume $r$ satisfies $\big(\f{d}{dt}-\De\big)r\geq 0$. Let $R>0$,
such that $\g(\ol{M}_{0,R})\subset \V\subset\grs{n}{m}$.
 Define $\varphi=R^2-r$ and $\varphi_+$ denotes the positive part of
$\varphi$. $h:\V\to\ir{}$ is a smooth positive function such that
\begin{equation}\label{m}
\Hess(h)\geq Ch^{-1}dh\otimes dh
\end{equation}
with $C\geq \f{3}{2}$. Then we have the estimate
$$\td{h}\varphi_+^2\leq \sup_{\ol{M}_{0,R}}\td{h}\varphi_+^2,$$
where $\td{h}=h\circ \g$.
\end{lem}

\begin{proof} Denote $\eta=\varphi_+^2$, then at an arbitrary
interior point of the support of $\varphi_+$, we have
\begin{equation}\label{n}
\eta'\leq 0,\ \eta^{-1}\big(\eta'\big)^2=4,\mbox{ and }\eta''=2,
\end{equation}
where $'$ denotes differentiation with respect to $r$. By (\ref{m}),
(\ref{j}), we have
\begin{equation}\label{d1}
\big(\f{d}{dt}-\De\big)\td{h}\leq -C{\td{h}}^{-1}|\n \td{h}|^2
\end{equation}
and moreover
\begin{eqnarray}
&&\big(\f{d}{dt}-\De\big)(\td{h}\eta)\nonumber\\
&=&\big(\f{d}{dt}-\De\big)\td{h}\cdot \eta+\td{h}\D\eta-2\n\td{h}\cdot \n\eta\nonumber\\
&\leq&-C{\td{h}}^{-1}|\n \td{h}|^2\eta+\td{h}\Big(\eta'\D r-\eta''|\n r|^2\Big)-2\n\td{h}\cdot \n\eta\nonumber\\
&\leq&-C{\td{h}}^{-1}|\n \td{h}|^2\eta-2\td{h}|\n
r|^2-2\n\td{h}\cdot \n\eta.\label{o}
\end{eqnarray}
Observe that
\begin{eqnarray}
-2\n\td{h}\cdot \n\eta&=&(2C-2)\n\td{h}\cdot \n\eta-2C\n\td{h}\cdot \n\eta\nonumber\\
&=&(2C-2)\eta^{-1}\big(\n(\td{h}\eta)-\td{h}\n \eta\big)\cdot \n \eta-2C \n\td{h}\cdot \n\eta\nonumber\\
&\leq&(2C-2)\eta^{-1}\n \eta\cdot
\n(\td{h}\eta)-(2C-2)\td{h}\eta^{-1}|\n
\eta|^2\nonumber\\
&&\hskip1.8in+C\td{h}^{-1}|\n \td{h}|^2\eta
+C\td{h}\eta^{-1}|\n \eta|^2\nonumber\\
&=&(2C-2)\eta^{-1}\n \eta\cdot \n(\td{h}\eta)+C\td{h}^{-1}|\n
\td{h}|^2\eta+(8-4C)\td{h}|\n r|^2.\label{p}
\end{eqnarray} Here (\ref{n}) has been used.
Substituting (\ref{p}) into (\ref{o}) gives
\begin{equation}
\D(\td{h}\eta)\leq (2C-2)\eta^{-1}\n \eta\cdot
\n(\td{h}\eta)+(6-4C)\td{h}|\n r|^2
\end{equation}
on the support of $\varphi_+$, The weak parabolic maximal principle
then implies the result. \end{proof}

\begin{lem}\label{l2}

Assume $r$ satisfies $(\f{d}{dt}-\De)r\geq 0$. If $\g(M_t)\subset
\V$ for arbitrary $t\in [0,T]$ ($T>0$), $h:\V\ra \R$ is a smooth
positive function satisfying (\ref{m}) with $C\geq 1$, then for
arbitrary $a\geq 0$, the following estimate holds.
\begin{equation}\label{es2}
\sup_{M_t}\td{h}(1+r)^{-a}\leq \sup_{M_0}\td{h}(1+r)^{-a}.
\end{equation}

\end{lem}

\begin{proof}

By $(\f{d}{dt}-\De)r\geq 0$,
\begin{equation}\aligned
(\f{d}{dt}-\De)(1+r)^{-a}&=-a(1+r)^{-a-1}(\f{d}{dt}-\De)r-a(a+1)(1+r)^{-a-2}|\n
r|^2\\
&\leq -a(a+1)(1+r)^{-a-2}|\n r|^2.
\endaligned
\end{equation}
In conjunction with (\ref{d1}), we have
\begin{equation}\aligned
&(\f{d}{dt}-\De)\big[\td{h}(1+r)^{-a}\big]\\
\leq&-C\td{h}^{-1}(1+r)^{-a}|\n
\td{h}|^2-a(a+1)\td{h}(1+r)^{-a-2}|\n r|^2-2\n{\td{h}}\cdot
\n(1+r)^{-a}\\
=&-C\td{h}^{-1}(1+r)^{-a}|\n \td{h}|^2-a(a+1)\td{h}(1+r)^{-a-2}|\n
r|^2+2a\n\td{h}\cdot (1+r)^{-a-1}\n r.
\endaligned
\end{equation}
$C\geq 1$ implies $Ca(a+1)\geq a^2$, then by Young's inequality,
$$(\f{d}{dt}-\De)\big[\td{h}(1+r)^{-a}\big]\leq 0.$$
Hence (\ref{es2}) follows from maximal principle for parabolic
equations on complete manifolds (see \cite{E-H1}).

\end{proof}

\begin{thm}\label{t1}
If the initial submanifold is an entire graph over $\ir{n}$,
i.e., $M_0=graph\ f_0$, where $f_0=(f_0^1,\cdots,f_0^m)$,
$f_0^\a=f_0^\a(x^1,\cdots,x^n)$; and
$$\De_{f_0}< 2,$$
where
$$\De_f(x)=\Big[\det\big(\de_{ij}+\f{\p f^\a}{\p x^i}(x)\f{\p f^\a}{\p x^j}(x)\Big]^{1/2}.$$
Then the submanifolds under the MCF are still entire graphs over the
same hyperplane, i.e., $M_t=graph\ f_t$; and
$$\De_{f_t}<2.$$
Moreover, if $(2-\De_{f_0})^{-1}$ has growth
\begin{equation*}
(2-\De_{f_0})^{-1}(x)\leq C_0 (|x|^2+1)^a
\end{equation*}
where $C_0, a$ are both positive constants, then the growth of
$(2-\De_{f_t})^{-1}$ can be controlled by
$$(2-\De_{f_t})^{-1}\leq 2C_0(|x|^2+2nt+1)^a.$$
\end{thm}

\begin{proof}
Define $h=v^{\f{3}{2}}(2-v)^{-\f{3}{2}}$, then on $\{P:v(P)<2\}$,
we have (see \cite{X-Y}, inequality (4.6))
\begin{eqnarray}\label{k}
\Hess(h)&=&h'\Hess(v)+h''dv\otimes dv\nonumber\\
&\geq&3 h g+\f{3}{2}h^{-1}dh\otimes dh.
\end{eqnarray}

Define $r(x,t)=|F|^2+2nt$, then $\D r=0$ . Hence, the estimate in
Lemma \ref{l1} holds. For arbitrary $x_0\in M_{t_0}$, choose
$R>0$, such that $r(x_0,t_0)<R^2$, then $\varphi_+(x_0,t_0)>0$
and Lemma \ref{l1} implies
\begin{equation}
\td{h}(x_0,t_0)\leq
\f{1}{\varphi_+(x_0,t_0)}\sup_{\ol{M}_{0,R}}\td{h}\varphi_+^2<+\infty.
\end{equation}
Noting that $\td{h}\ra +\infty$ when $v\ra 2_-$ we have
$v(x_0,t_0)<2$ and the first result follows.   For 
$\mathbf{x}\in\ir{n}$, it is not difficult to see that 
$$(\f{d}{dt}-\De)(|\mathbf{x}|^2+2nt)\geq 0.$$ 
Now, we define $r=|\mathbf{x}|^2+2nt$, then the second assertion  
easily follows from Lemma \ref{l2}.
\end{proof}

Choose $$h=\sec^2(\sqrt{2}\rho)$$ and by the similar argument we 
can improve the previous result of the first author  \cite{X3} as 
follows
\begin{thm}\label{t4}
If the Gauss image of the initial complete submanifold $M_0$ is
contained in an open geodesic ball of the radius
$R_0\le\frac{\sqrt{2}}{4}\pi$ in $\grs{m}{n}$, then the Gauss images
of all the submanifolds under the MCF are also contained in the same
geodesic ball. Moreover, if
$$(\f{\sqrt{2}}{4}\pi-\rho)^{-1}\leq C_0\big(|F|^2+1\big)^a\qquad \mbox{on }M_0,$$
(Here $\rho$ denotes the distance function on $\grs{n}{m}$ from the
center of the geodesic ball, $C_0, a$ are both positive constants.)
then
$$(\f{\sqrt{2}}{4}\pi-\rho)^{-1}\leq 2C_0(|F|^2+2nt+1)^a$$
\label{cp} for arbitrary $a\ge 0$.
\end{thm}

Let $M \to \ir{4}$ be a  surface. Let $\pi_1 :\grs{2}{2} \to S^2$
be the projection of $\grs{2}{2}$ into its first factor, and
$\pi_2$ be the projection into the second factor. Define
$\gamma_i=\pi_i\circ\gamma.$ We also have

\begin{thm}
If the partial Gauss image of an initial surface $M$ in $\ir{4}$
is contained in a  hemisphere, then the partial Gauss image of
all the surfaces under MCF are  same hemisphere.
\end{thm}

\Section{Curvature estimates}{Curvature estimates}
\medskip

Let $h:\V\to\ir{}$ be a smooth function defined on an open subset
$\V\subset \grs{n}{m}$, and $h\geq 1$. Suppose that $\Hess(h)$ is nonnegative
definite on $\V$ and have the estimate
\begin{equation}\label{k}
\Hess(h)\geq 3 h g+ \f{3}{2} h^{-1}dh\otimes dh,
\end{equation}
where $g$ is the metric tensor on $\grs{n}{m}$.  $r$ is a smooth, 
non-negative function on $\R^{n+m}\times \R$ satisfying
\begin{equation}\label{co3}
\Big|\D r\Big|\leq C(n)\qquad \mbox{and}\qquad|\n r|^2\leq C(n)r.
\end{equation}

\begin{thm}\label{t2}

Let $R>0,T>0$ be such that for any $x\in \ol{M}_{t,R}$, where
$t\in [0,T]$, we have $\g(x)\in \V$. Then for any $t\in [0,T]$
and $\th\in [0,1)$, we have the estimate
\begin{equation*}\label{cu0}
\sup_{x\in \ol{M}_{t,\th R}}|B|^2\leq
C(n)(1-\th^2)^{-2}(t^{-1}+R^{-2})\sup_{x\in \ol{M}_{s,R},s\in
[0,t]}\td{h}^2,
\end{equation*}
where $\td h=h\circ\g$.
\end{thm}

The proof of Theorem \ref{t2} shall be given later. At first we will see several
applications of it.

Let $r=|\mathbf{x}|^2$ for $\mathbf{x}\in \ir{n}$, then
$$\Big|\D r\Big|=\Big|2x^i\D x^i-2|\n x^i|^2\Big|\leq 2n$$
and
$$|\n r|^2=|2x^i \n x^i|^2=4(x^i)^2|\n x^i|^2\leq 4r.$$
Hence Theorem \ref{t2} yields

\begin{cor}\label{co1}

Let $R>0, T>0$ be such that for any $t\in [0,T]$, $M_t\cap
\big((B_R\subset \R^n)\times \R^m\big)$ is a graph over $B_R$, i.e.
$M_t\cap \big((B_R\subset \R^n)\times \R^m\big)=\{(x,f_t(x)):x\in
B_R\}$, and $\De_{f_t}<2$, then the following estimate holds for
arbitrary $t\in [0,T]$ and $\th\in [0,1)$
$$\sup_{(x,f_t(x))\in K(t,\th R)}|B|^2\leq
C(n)(1-\th^2)^{-2}(t^{-1}+R^{-2})\sup_{s\in [0,t]}\sup_{(x,f_s(x))\in K(s,R)}(2-\De_{f_s})^{-3}.$$
Here
$$K(s,R)=\{(x,f_s(x)):x\in B_R\}.$$
\end{cor}

Combing Corollary \ref{co1} and Theorem \ref{t1} yields
\begin{cor}
If the initial submanifold is an entire graph over $\R^n$, i.e.
$M_0=graph f_0$, and $\De_{f_0}<2$,
$(2-\De_{f_0})^{-1}=o(|x|^{2a})$, then we have the estimate
$$\sup_{(x,f_t(x))\in K(t,\th R)}|B|^2\leq
C(n)(1-\th^2)^{-2}(t^{-1}+R^{-2})(R^2+2nt+1)^{3a}.$$
Here $\th\in [0,1)$ and the denotation of $K(\ ,\ )$ is similar to Corollary \ref{co1}.
\end{cor}

Similarly, if
$$r=|\mathbf{x}|^2+2nt,$$
then it is easy to check that $r$ satisfies (\ref{co3}). Applying
Theorem \ref{t2} and Theorem \ref{t4} we have

\begin{cor}\label{co4}

Let $R>0, T>0$ be such that for any $t\in [0,T]$, if $x\in M_t$
satisfies $|F|^2+2nt\leq R^2$, then $\g(x)$ lies in an open geodesic ball
centered at a fixed point $P_0$ of radius $\f{\sqrt{2}}{4}\pi$ in $\grs{n}{m}$. Then
the following estimate holds for arbitrary $t\in [0,T]$ and $\th\in
[0,1)$
$$\sup_{x\in K(t,\th R)}|B|^2\leq
C(n)(1-\th^2)^{-2}t^{-1}\sup_{0\leq s\leq t}\sup_{x\in K(s,R)}(\f{\sqrt{2}}{4}\pi-\rho)^{-3}.$$
Here
$$K(s,R)=\{x\in M_s:|F|^2+2ns\leq R^2\}.$$
\end{cor}

\begin{cor}
If the Gauss image of the initial complete submanifold $M_0$ is contained in an open
geodesic ball of radius $\f{\sqrt{2}}{4}\pi$ in $\grs{n}{m}$, and $(\f{\sqrt{2}}{4}\pi-\rho)^{-1}$
has growth
$$(\f{\sqrt{2}}{4}\pi-\rho)^{-1}=o(|F|^{2a}),$$
then we have the estimate
$$\sup_{x\in K(t,\th R)}|B|^2\leq C(n)(1-\th^2)^{-2}t^{-1}(R^2+1)^{3a}.$$
Here $\th\in [0,1)$ and the denotation of $K(\ ,\ )$ is similar to Corollary \ref{co4}.
\end{cor}

\begin{rem}

When $x\in K(t,\th R)$,
$$2nt\leq |F|^2+2nt\leq \th^2 R^2\leq R^2,$$
so
$$R^{-2}\leq \f{1}{2n}t^{-1}.$$
Hence in the process of applying Theorem \ref{t2} to Corollary \ref{co4}, $t^{-1}+R^{-2}$ could be replaced by $t^{-1}$.
\end{rem}

\emph{Proof of Theorem \ref{t2}.}
Let $\varphi=\varphi(\td{h})$ be a smooth nonnegative function of
$\td{h}$ to be determined later, and $'$ denotes derivative with
respect to $\td{h}$, then  from  (\ref{B}), (\ref{j}) and (\ref{k})
we have
\begin{equation}
\aligned
\D|B|^2 \varphi
&=\D |B|^2\cdot\varphi+|B|^2\D \varphi-2\n|B|^2\cdot \n\varphi\\
&\leq\big(-2\big|\n
|B|\big|^2+3|B|^4\big)\varphi\\
&\hskip0.5in+|B|^2\big(\varphi'\D\td{h}-\varphi''|\n \td{h}|^2\big)
-2\n|B|^2\cdot \n\varphi\\
&\leq\big(-2\big|\n
|B|\big|^2+3|B|^4\big)\varphi-|B|^2\varphi'(3\td{h}|B|^2+\f{3}{2}\td{h}^{-1}|\n
\td{h}|^2)\\
&\hskip1.2in-|B|^2\varphi''|\n \td{h}|^2-2\n|B|^2\cdot \n\varphi\\
\endaligned\label{cu1}
\end{equation}
The last term can be estimated by
\begin{eqnarray}\label{cu2}
-2\n|B|^2\cdot \n\varphi&=&-\n|B|^2\cdot \n\varphi-\n|B|^2\cdot \n\varphi\nonumber\\
&=&-\varphi^{-1}\big(\n(|B|^2\varphi)-|B|^2\n\varphi\big)\cdot \n\varphi-2|B|\n |B|\cdot\n\varphi\nonumber\\
&\le&-\varphi^{-1}\n\varphi\cdot\n(|B|^2\varphi)+|B|^2\varphi^{-1}|\n \varphi|^2\nonumber\\
&&\hskip1in+2\big|\n |B|\big|^2\varphi+\f{1}{2}|B|^2\varphi^{-1}|\n \varphi|^2\nonumber\\
&=&-\varphi^{-1}\n\varphi\cdot\n(|B|^2\varphi)+2\big|\n
|B|\big|^2\varphi+\f{3}{2}|B|^2\varphi^{-1}|\n \varphi|^2.
\end{eqnarray}
Substituting (\ref{cu2}) into (\ref{cu1}) gives
\begin{eqnarray}\label{cu3}
\D |B|^2\varphi &\leq&-(3\varphi'\td{h}-3\varphi)|B|^4\nonumber\\
&-&\big(\f{3}{2}\varphi'\td{h}^{-1}+\varphi''
-\f{3}{2}\varphi^{-1}(\varphi')^2\big)|B|^2|\n\td{h}|^2-\varphi^{-1}\n\varphi\cdot \n(|B|^2\varphi).\nonumber\\
\end{eqnarray}
Now we let $\varphi(\td{h})=\f{\td{h}}{1-k\td{h}}$, $k\geq 0$ to
be chosen; then
\begin{eqnarray}
&&3\varphi'\td{h}-3\varphi=3k\varphi^2,\\
&&\f{3}{2}\varphi'\td{h}^{-1}+\varphi''-\f{3}{2}\varphi^{-1}(\varphi')^2=\f{k}{2\td{h}(1-k\td{h})^2}\varphi,\\
&&\varphi^{-1}\n\varphi=\f{1}{\td{h}(1-k\td{h})}\n\td{h}.
\end{eqnarray}
Substituting these identities into (\ref{cu3}) we derive for
$g=|B|^2\varphi$ the inequality
\begin{equation}
\D g\leq -3k
g^2-\f{k}{2\td{h}(1-k\td{h})^2}|\n\td{h}|^2g-\f{1}{\td{h}(1-k\td{h})}\n\td{h}\cdot
\n g.
\end{equation}
As in Lemma \ref{l1}, we define $\eta=(R^2-r)_+^2$, then on the
support of $\eta$,
\begin{eqnarray*}
\D\eta&=&-2(R^2-r)\D r-2|\n r|^2\\
&\leq&2C(n)R^2-2|\n r|^2
\end{eqnarray*}
and
\begin{eqnarray}\label{cu4}
\D g\eta&=&\D g\cdot \eta+g\D \eta-2\n g\cdot\n\eta\nonumber\\
&\leq&-3k g^2\eta-\f{k}{2\td{h}(1-k\td{h})^2}|\n\td{h}|^2g\eta
-\f{1}{\td{h}(1-k\td{h})}\n\td{h}\cdot \n g\cdot\eta\nonumber\\
&&\qquad +2C(n)R^2g-2g|\n r|^2-2\n g\cdot\n \eta;
\end{eqnarray}
where
\begin{eqnarray}\label{cu5}
-2\n g \cdot\n\eta&=&-2\eta^{-1}\n\eta\cdot\n(g\eta)+2g\eta^{-1}|\n\eta|^2\nonumber\\
&=&-2\eta^{-1}\n\eta\cdot\n(g\eta)+8g|\n r|^2
\end{eqnarray}
and
\begin{eqnarray}\label{cu6}
&&-\f{1}{\td{h}(1-k\td{h})}\n\td{h}\cdot \n g\cdot\eta\nonumber\\
&=&-\f{1}{\td{h}(1-k\td{h})}\n\td{h}\cdot \n (g\eta)+\f{1}{\td{h}(1-k\td{h})}\n\td{h}\cdot  g\n\eta\nonumber\\
&\leq&-\f{1}{\td{h}(1-k\td{h})}\n\td{h}\cdot \n (g\eta)+\f{k}{2\td{h}(1-k\td{h})^2}|\n\td{h}|^2g\eta
+\f{1}{2k\td{h}}g\eta^{-1}|\n\eta|^2\nonumber\\
&=&-\f{1}{\td{h}(1-k\td{h})}\n\td{h}\cdot \n
(g\eta)+\f{k}{2\td{h}(1-k\td{h})^2}|\n\td{h}|^2g\eta+\f{2}{k\td{h}}g|\n
r|^2.
\end{eqnarray}
Substituting (\ref{cu5}) and (\ref{cu6}) into (\ref{cu4}) gives
\begin{eqnarray}\label{cu7}
\D g\eta&\leq&
-3kg^2\eta-\big(2\eta^{-1}\n\eta+\f{1}{\td{h}(1-k\td{h})}\n\td{h}\big)\cdot
\n(g\eta)\nonumber\\
&&\hskip1in+C(n)\big[(1+\f{1}{k\td{h}})r+R^2\big]g.
\end{eqnarray}
Furthermore
\begin{eqnarray}\label{cu8}
\D (tg\eta)&\leq&
-3ktg^2\eta-\big(2\eta^{-1}\n\eta+\f{1}{\td{h}(1-k\td{h})}\n\td{h}\big)\cdot
\n(tg\eta)\nonumber\\
&&\hskip1in+C(n)\big[(1+\f{1}{k\td{h}})r+R^2\big]t g+g\eta.
\end{eqnarray}
Denote
$$m(T)=\sup_{0\leq t\leq T}\sup_{\ol{M}_{t,R}}tg\eta=t_0g(x_0,t_0)\eta(x_0,t_0),$$
then $t_0>0$, $r(x_0,t_0)<R^2$ and hence
$$\D (tg\eta)\geq 0,\qquad \n(tg\eta)=0$$
at $(x_0,t_0)$. (\ref{cu8}) implies
$$3kt_0g^2\eta\leq C(n)\big[(1+\f{1}{k\td{h}})r+R^2\big]t_0 g+g\eta.$$
Multiplying by $\f{t_0\eta}{3k}$ yields
\begin{eqnarray*}
m(T)^2&\leq&\f{C(n)}{3k}(1+\f{1}{k\td{h}})R^2t_0^2g\eta+\f{t_0g\eta^2}{3k}\\
&\leq&\f{C(n)}{3k}\big((1+\f{1}{k\td{h}})R^2 T+\eta\big)m(T);
\end{eqnarray*}
by $\eta=(R^2-r)_+^2\leq R^4$ we arrive at
$$g\eta T\leq m(T)\leq \f{C(n)}{3k}\big((1+\f{1}{k\td{h}})R^2 T+R^4\big)$$
in $\ol{M}_{T,R}$. Let now
\begin{equation}
k=\f{1}{2}\inf_{x\in \ol{M}_{t,R},t\in [0,T]}\td{h}^{-1}.
\end{equation}
Since $\varphi=\f{\td{h}}{1-k\td{h}}\geq \f{1}{1-k}\geq 1$ (by $\td{h}\geq 1$) and $\eta\geq (1-\th^2)^2
R^4$ in $\ol{M}_{T,\th R}$, we have
\begin{equation}\label{cu9}
\sup_{x\in \ol{M}_{T,\th R}}|B|^2\leq
C(n)(1-\th^2)^{-2}(T^{-1}+R^{-2})\sup_{t\in
[0,T]}\sup_{x\in \ol{M}_{t,R}}\td{h}^2
\end{equation}
and finally (\ref{cu0}) follows from replacing $T$ by $t$,
replacing $t$ by $s$ in (\ref{cu9}). $\hfill\Box$

Substituting $\varphi=\td h$ into (\ref{cu3}) gives 

$$\D |B|^2\td h\le -\td h^{-1}\n\td h\cdot \n(|B|^2\\td h).$$
Using the parabolic maximum principle for complete manifolds in 
\cite{E-H1}, we have

\begin{cor}
Let $M$ be a complete $n-$submanifold in $\ir{m+n}$ with bounded 
curvature. Then
$$\sup_{M_t}|B|^2\td h\le \sup_{M_0}|B|^2\td h.$$
\end{cor}

\begin{rem}
When $\V$ is a geodesic ball of radius 
$\rho_0<\f{\sqrt{2}}{4}\pi$, we can choose 
$h=\sec^2(\sqrt{2}\rho)$. So the above estimate is an improvement 
of Thm. 4.2 in \cite{X3}.
\end{rem}

Furthermore, we can give a prior estimates for $|\n^m B|^2$ by 
induction.

\begin{thm}\label{t3}
Our denotation and assumption is similar to Theorem \ref{t2},
then for arbitrary $m\geq 0$, $\th\in [0,1)$ and $t\in [0,T]$, we
have the estimate
$$\sup_{x\in \ol{M}_{t,\th R}}|\n ^m B|^2\leq c_m(R^{-2}+t^{-1})^{m+1};$$
where $c_m=c_m(\th,n,\sup_{\ol{M}_{s,R},s\in [0,t]}\td{h})$.
\end{thm}

\begin{proof}
We proceed by induction on $m$. The case $m=0$ has been
established as Theorem \ref{t2}. Now we suppose the inequality
holds for $0\leq k\leq m-1$. Denote
$\psi(t)=(R^{-2}+t^{-1})^{-1}=\f{R^2 t}{R^2+t}$; we shall
estimate the upper bound of $\psi^{m+1}|\n^m B|^2$ on
$\ol{M}_{T,\th R}$ for fixed $\th\in [0,1)$.

By computing
\begin{eqnarray}\label{cu10}
&&\D\psi^{m+1}|\n^m B|^2 \leq-2\psi^{m+1}|\n^{m+1}
B|^2+\big(\f{d}{dt}\psi^{m+1}\big)|\n^m
B|^2\nonumber\\
&&\hskip1in+C(m,n)\psi^{m+1}\sum_{i+j+k=m,i\leq j\leq k}|\n^i
B||\n^j B||\n^k B||\n^m B|.
\end{eqnarray}
By inductive assumption
$$\sup_{x\in \ol{M}_{t,\f{1+\th}{2}R}}\psi^{k+1}|\n^k B|^2\leq c_k$$
for every $0\leq k\leq m-1$ and $t\in [0,T]$, where
$$c_k=c_k(\th,n,\sup_{x\in \ol{M}_{t,R},t\in [0,T]}\td{h})$$
(note that $c_k$ depends on $\f{1+\th}{2}$, which only depends on
$\th\in [0,1)$); which implies $|\n^i B|\leq
c_i^{1/2}\psi^{-(i+1)/2}$, $|\n^j B|\leq
c_j^{1/2}\psi^{-(j+1)/2}$; moreover
\begin{eqnarray}\label{cu11}
&&\psi^{m+1}\sum_{i+j+k=m,i\leq j\leq k}|\n^i B||\n^j B||\n^k B||\n^m B|\nonumber\\
&&\hskip1in\leq C\sum_{i+j+k=m,i\leq j\leq k}\psi^{\f{k+m}{2}}|\n^k B||\n^m B|\nonumber\\
&&\hskip1in\leq C\sum_{k\leq m}\psi^k|\n^k B|^2.
\end{eqnarray}
On the other hand
\begin{eqnarray}\label{cu12}
\f{d}{dt}\psi^{m+1}=(m+1)\psi^m\f{R^4}{(R^2+t)^2}\leq (m+1)\psi^m.
\end{eqnarray}
Substituting (\ref{cu11}) and (\ref{cu12}) into (\ref{cu10}) gives
\begin{eqnarray}
\D \psi^{m+1}|\n^m B|^2\leq -2\psi^{m+1}|\n
^{m+1}B|^2+C\sum_{k\leq m}\psi^k|\n^k B|^2
\end{eqnarray}
on $\ol{M}_{t,\f{1+\th}{2}R}$ for arbitrary $t\in [0,T]$; where
$$C=C(\th,n,\sup_{x\in \ol{M}_{t,R},t\in [0,T]}\td{h}).$$
Now we define $f=\psi^{m+1}|\n^m B|^2(\La+\psi^m|\n^{m-1} B|^2)$,
where $\La>0$ to be chosen later. By computing
\begin{eqnarray}
\D f&\leq&-2\pB{m+1}{m+1}(\La+\pB{m}{m-1})\nonumber\\
&&+C\sum_{k\leq m}\pB{k}{k}(\La+\pB{m}{m-1})\nonumber\\
&&-2\psi^{2m+1}|\n^m B|^4+C\sum_{k\leq m-1}\pB{k}{k}\pB{m+1}{m}\nonumber\\
&&-2\psi^{2m+1}\n|\n^m B|^2\cdot \n|\n^{m-1}B|^2;
\end{eqnarray}
where the last term can be estimated by
\begin{eqnarray}
&&-2\psi^{2m+1}\n |\n^m B|^2\cdot \n|\n^{m-1}B|^2\nonumber\\
&&\hskip0.5in =-8\psi^{2m+1}|\n^m B|\n|\n^m B|\cdot |\n^{m-1}B|\n|\n^{m-1}B|\nonumber\\
&&\hskip0.5in \leq 2\pB{m+1}{m+1}(\La+\pB{m}{m-1})+8\psi^{2m+1}|\n^m B|^4\f{\pB{m}{m-1}}{\La+\pB{m}{m-1}}\nonumber\\
&&\hskip1in\leq
2\pB{m+1}{m+1}(\La+\pB{m}{m-1})+\f{8c_{m-1}}{\La+c_{m-1}}\psi^{2m+1}|\n^m
B|^4.
\end{eqnarray}
Hence we derive
\begin{eqnarray}
\D f&\leq&-(2-\f{8c_{m-1}}{\La+c_{m-1}})\psi^{-1}\big(\pB{m+1}{m}\big)^2\nonumber\\
&&+C\psi^{-1}\big(\sum_{k\leq m}\pB{k+1}{k}(\La+\pB{m}{m-1})\nonumber\\
&&\hskip1in+\sum_{k\leq m-1}\pB{k+1}{k}\pB{m+1}{m}\big).
\end{eqnarray}
Now we let $\La=7c_{m-1}+1$, then
$$\D f\leq -\psi^{-1}(\La+\pB{m}{m-1})^{-2}f^2+C\psi^{-1}(1+f);$$
by Young's inequality,
\begin{eqnarray*}
Cf&\leq&\f{1}{2}(\La+\pB{m}{m-1})^{-2}f^2+\f{1}{2}C^2(\La+\pB{m}{m-1})^2\\
&\leq&\f{1}{2}(\La+\pB{m}{m-1})^{-2}f^2+\f{1}{2}C^2(8c_{m-1}+1)^2;
\end{eqnarray*}
hence we have
\begin{eqnarray}
\D f\leq -\psi^{-1}(\de f^2-C);
\end{eqnarray}
where
$$\de=\f{\big(C(8c_{m-1}+1)^2-1\big)^2}{2(8c_{m-1}+1)^2}>0$$
and $C$ is a positive constant depending on
$n,m,\sup_{\ol{M}_{t,R},t\in [0,T]}\td{h}$.

Now we let $\varphi=\big(\f{1+\th}{2}R\big)^2-r$, and
$\eta=(\varphi_+)^2$, then $\eta$ is a nonnegative function which
vanishes outside $\ol{M}_{t,\f{1+\th}{2}R}$; similar to
(\ref{cu4}), we can derive
\begin{eqnarray}
\D f\eta\leq \psi^{-1}(\de f^2-C)\eta+C(n)R^2
f-2\eta^{-1}\n\eta\cdot \n(f\eta)
\end{eqnarray}
on $\ol{M}_{t,\f{1+\th}{2}R}$. Denote $m(T)=\max_{0\leq t\leq
T}\max_{x\in \ol{M}_{t,\f{1+\th}{2}R}}f\eta=f\eta(x_0,t_0)$, we
have
$$f^2 \eta\leq \f{1}{\de}\big(C\eta+C(n)R^2f\psi\big).$$
Multiplying by $\eta$, using $\eta\leq R^4$, $\psi\leq R^2$, we
have
\begin{eqnarray*}
f^2\eta^2&\leq&\f{1}{\de}(C\eta^2+C(n)R^2f\eta\psi)\leq \f{1}{\de}\big(CR^8+C(n)R^4f\eta\big)\\
&\leq&\f{1}{\de}\big(CR^8+\f{\de}{2}f^2\eta^2+\f{C(n)^2R^8}{2\de}\big)
\end{eqnarray*}
i.e., $m(T)^2=f^2\eta^2\leq CR^8$,
$$\sup_{0\leq t\leq T}\sup_{x\in \ol{M}_{t,\f{1+\th}{2}R}}f\eta\leq CR^4;$$
where $C=C(\th,n,m,\sup_{\ol{M}_{t,R},t\in [0,T]}\td{h})$.

Finally, since $\eta=\big((\f{1+\th}{2}R)^2-(\th
R)^2\big)^2=\f{1+2\th-3\th^2}{4}R^4$ on $\ol{M}_{T,R}$ and
$\La+\pB{m}{m-1}\geq 7c_{m-1}+1$, we have
\begin{eqnarray}
\sup_{x\in\ol{M}_{T,\th R}}\pB{m+1}{m}\leq c_m(\th,n,\sup_{x\in
\ol{M}_{t,R},t\in [0,T]}\td{h}).
\end{eqnarray}
and the conclusion follows from replacing $T$ by $t$ and
replacing $t$ by $s$.
\end{proof}

\bibliographystyle{amsplain}

\end{document}